\documentclass[11pt,oneside, a4paper,reqno]{amsart}
\usepackage[utf8]{inputenc}
\usepackage[T1]{fontenc}
\usepackage{lmodern}
\usepackage{etoolbox}

\usepackage{geometry}
\usepackage[all]{xy}
\pdfpagewidth=\paperwidth
\pdfpageheight=\paperheight
\usepackage{xcolor}

\usepackage{amsmath,amsfonts,amssymb,mathtools,manfnt,amsthm}

\newtheorem{thm}{Theorem}[section]
\newtheorem{cor}[thm]{Corollary}
\newtheorem{lem}[thm]{Lemma}
\newtheorem{prop}[thm]{Proposition}

\theoremstyle{definition}
\newtheorem{dfn}[thm]{Definition}

\theoremstyle{remark}
\newtheorem{rmk}[thm]{Remark}

\newtheorem{example}[thm]{Example}

\usepackage{tikz}
\usepackage{manfnt}
\usetikzlibrary{matrix}
\usetikzlibrary{arrows}
\usetikzlibrary{positioning}
\usetikzlibrary{calc}
\usepackage{tikz-cd}

\usepackage{enumitem}

\def\namedlabel#1#2{\begingroup
	#2%
	\def\@currentlabel{#2}%
	\phantomsection\label{#1}\endgroup
}
\makeatother

\makeatletter
\g@addto@macro\bfseries{\boldmath}
\makeatother

\usepackage{datetime}
\usepackage{url}

\usepackage[unicode=true,colorlinks=true,linkcolor=black,citecolor=black,urlcolor=black,breaklinks=true]{hyperref}

\DeclareMathOperator{\id}{id}

\DeclareMathOperator{\gr}{Gr}

\newcommand\restr[2]{{
		\left.\kern-\nulldelimiterspace 
		#1 
		\vphantom{|} 
		\right|_{#2} 
}}

\newcommand{\ol}[1]{\overline{#1}}

\renewcommand{\AA}{\mathbb{A}}
\newcommand{\BB}{\mathbb{B}}

\newcommand{\DD}{\mathbb{D}}

\newcommand{\NN}{\mathbb{N}}

\newcommand{\RR}{\mathbb{R}}

\newcommand{\Kk}{\mathcal{K}}

\usepackage[
backend=bibtex,
style=ieee-alphabetic,
citestyle=alphabetic,
doi=true,
isbn=true,
sorting = nyt,
maxbibnames=99
]{biblatex}

\AtEveryBibitem{\clearfield{issn}}
\AtEveryCitekey{\clearfield{issn}}

\defbibheading{bibliography}[\refname]{}

\usepackage[unicode=true,colorlinks=true,linkcolor=black,citecolor=black,urlcolor=black,breaklinks=true]{hyperref}

\bibliography{References} 

\usepackage[symbol]{footmisc}

\begin{document}
  
\title{A closed graph theorem for hyperbolic iterated function systems}
\author{Alexander Mundey}

\begin{abstract}
	In this note we introduce a notion of a morphism between two hyperbolic iterated function systems. We prove that the graph of a morphism is the attractor of an iterated function system, giving a Closed Graph Theorem, and show how it can be used to approach the topological conjugacy problem for iterated function systems. 
\end{abstract}

\maketitle

\section{Introduction}

Since Hutchinson's seminal paper \cite{Hut81}, iterated function systems have remained close to the heart of fractal geometry. Iterated function systems have continued to be studied and generalised in numerous directions.

In this article we focus on the dynamics of hyperbolic iterated function systems from a topological viewpoint, rather than  geometric or measure theoretic perspectives.
Although we work in the hyperbolic setting, we take a viewpoint similar to that of topological iterated function systems \cite{BWL14,Kam93,Kie02}.
 Determining whether two iterated function systems are topologically conjugate (in the sense of Definition~\ref{dfn:morphism} below) is a subtle problem, as highlighted in Example~\ref{ex:intervalmaps}. In contrast to symbolic dynamical systems, the topology of the space and attractor play a greater role determining conjugacy for iterated function systems. 

The main results of this article---Theorem~\ref{thm:graphattractor} and Corollary~\ref{cor:closedgraph}---establish a Closed Graph Theorem for an elementary notion of morphism between hyperbolic iterated function systems. Closed Graph Theorems are prevalent throughout mathematics with the most well-known results being those for continuous maps between Banach spaces and for continuous maps between compact Hausdorff spaces. 
In the latter case, the Closed Graph Theorem states that a continuous function $f \colon X \to Y$ between compact Hausdorff spaces is closed if and only its graph $\gr(f) = \{(x,f(x)) \mid x \in X\}$ is a closed subspace of $X \times Y$. An alternate formulation---which is perhaps more relevant to our context---is that $f \colon X \to Y$ is continuous if and only if $\gr(f)$ is itself a compact Hausdorff space. Applying a set-theoretic lens to functions---where $f$ is defined as its graph---one can interpret $f$ itself as a compact Hausdorff space. 

In Theorem~\ref{thm:graphattractor}, we show that the graph of a morphism between two iterated function systems is itself the attractor of an associated iterated function system. We show that by using this Closed Graph Theorem it is sometimes possible to deduce that two iterated function systems are not topologically conjugate. We also obtain a  characterisation of the code map from a labelled Cantor space as the attractor of a certain iterated function system. 
 
 \subsection*{Acknowledgements}
 The author would like to thank Adam Rennie for helpful discussions and careful proofreading. This research was supported by an Australian Government Research Training Program (RTP) Scholarship.
 
\section{Iterated function systems}
Although many generalisations of iterated function systems exist in the literature, in this article we restrict our attention entirely to the setting of hyperbolic systems. That is iterated function systems consisting of contractions. 
\begin{dfn}
	A {\emph{(hyperbolic) iterated function system}} $(X,\Gamma)$ consists of a complete metric space $(X,d)$ together with a finite collection $\Gamma$ of proper contractions on $X$. That is for each $\gamma \in \Gamma$ there exists $0 \le c_\gamma < 1$ such that $d(\gamma(x),\gamma(y)) \le c_\gamma d(x,y)$ for all $x,\,y \in X$. 
\end{dfn}

Denote by $\Kk(X)$ the collection of non-empty compact subsets of $X$.
Recall that the Hausdorff metric $d_H$ on $\Kk(X)$ is defined by
\[
d_H(A,B) = \max \Big\{ \sup_{x \in A} \inf_{y \in B} d(x,y), \,\sup_{y \in B} \inf_{x \in A} d(x,y) \Big\}
\]
for all $A,\,B \in \Kk(X)$. It is well-known  that if $(X,d)$ is a complete metric space then so is $(\Kk(X),d_H)$ (see for example \cite[Proposition 1.1.5]{Kig01}).

Given an iterated function system $(X,\Gamma)$ we abuse notation and also use $\Gamma$ to denote the \emph{Hutchinson operator} $\Gamma \colon \Kk(X) \to \Kk(X)$ defined, for all $K \in \Kk(X)$, by
\[
\Gamma(K) = \bigcup_{\gamma \in \Gamma} \gamma(K).
\]
Hutchinson \cite{Hut81} showed that $\Gamma$ is a contraction on $(\Kk(X),d_H)$, and consequently has a unique fixed point $\AA \in \Kk(X)$ by the Contraction Mapping Principle. The fixed-point $\AA$ is called the \emph{attractor} of $(X,\Gamma)$. In particular, $\Gamma(\AA) = \AA$ and for any $K \in \Kk(X)$ we have $d_H(\Gamma^k(K), \AA) \to 0$ as $k \to \infty$. 
This result is collectively referred to as Hutchinson's Theorem. Hutchinson's original result was stated for $X = \RR^n$, however the proof of the result for general hyperbolic systems remains nearly identical (cf. \cite[Theorem 1.1.7]{Kig01}).

A non-empty compact subset $K$ of $X$ is said to be \emph{backward invariant} if $K \subseteq \Gamma(K)$. Backward invariant sets were also called \emph{sub-self-similar sets} by Falconer~\cite{Fal95}. In hyperbolic iterated function systems, backward invariant sets are necessarily contained in the attractor. We require the following lemma in the sequel.
\begin{lem}\label{lem:backinvariant}
	Let $(X,d)$ be a complete metric space. Suppose that $(K_i)_{i=1}^{\infty}$ is a sequence in $\Kk(X)$ such that $K_i \subseteq K_{i+1}$ for all $i \in \NN$ and $d_H(K_i,K) \to 0$ for some $K \in \Kk(X)$.  Then $K_i \subseteq K$ for all $i \in \NN$.  In particular, if $(X,\Gamma )$ is a hyperbolic iterated function system with attractor $\AA$ and $K \in \Kk(X)$ is backward invariant, then $K \subseteq \AA$.
\end{lem}

\begin{proof}
	Suppose for contradiction that there exists $i \in \NN$ and $x_0 \in K_i \setminus K$. Since $K$ is compact $\inf_{y \in K} d(x_0,y) > 0.$
	As $x_0 \in K_i \subseteq K_j$ for all $j \ge i$ it follows that
	\[
	d_H(K_j,K) = \max \Big\{ \sup_{x \in K_j} \inf_{y \in K} d(x,y), \,\sup_{y \in K} \inf_{x \in K_j} d(x,y) \Big\} \ge \inf_{y \in K} d(x_0,y).
	\]
	As $K$ is compact, $\inf_{y \in K} d(x_0,y)$ is strictly positive. This contradicts that $d_H(K_i,K) \to 0$.
	
	For the second statement note that if $K \subseteq \Gamma(K)$, then $\Gamma^k(K) \subseteq \Gamma^{k+1}(K)$ for all $k \in \NN$. 
	Hutchinson's Theorem implies that $d_H(\Gamma^k(K),\AA) \to 0$, so the second statement follows from the first.
\end{proof}

\section{Morphisms of iterated function systems}

Topological conjugacy of iterated function systems is a well-established notion and the problem of determining whether two systems are topologically conjugate can be approached in numerous ways (see for example \cite[Corollary 1.27]{Kam00}). 
When generalising conjugacy there are choices to be made about selecting an appropriate notion of morphism, and several approaches have been taken previously. Kieninger, for instance, describes semiconjugacy of topological iterated function systems  \cite[Definition 4.6.3]{Kie02}, which is analogous to the corresponding notion in symbolic dynamics. Another approach is via the fractal homeomorphisms of Barnsley \cite{Bar09} which use shift invariant sections of a code map. 

Expanding on the established definition of conjugacy we use the following---somewhat na\"ive---notion of morphism, related to Kieninger's semiconjugacies.

\begin{dfn} \label{dfn:morphism}
	A \emph{(topological) morphism} from an iterated function system $(X,\Gamma)$ to $(Y,\Lambda)$ is a pair $(f,\alpha)$ consisting of a continuous map $f \colon X \to Y$ and a function $\alpha \colon \Gamma \to \Lambda$ such that for each $\gamma \in \Gamma$ the diagram
	\[
		\begin{tikzcd}
		X \arrow[r,"\gamma"] \arrow[d, "f"'] & X \arrow[d, "f"] \\
		Y \arrow[r,"\alpha(\gamma)"] & Y  
		\end{tikzcd}
	\]
	commutes. We write $(f,\alpha) \colon (X,\Gamma) \to (Y,\Lambda)$ to mean that $(f,\alpha)$ is a morphism from $(X,\Gamma)$ to $(Y,\Lambda)$. We mention some special types of morphism. 
	\begin{enumerate}
		\item $(f,\alpha)$ is an \emph{embedding} if both $f$ and $\alpha$ are injective.
		\item $(f,\alpha)$ is a \emph{semiconjugacy} if both $f$ and $\alpha$ are surjective.
		\item$(f,\alpha)$ is an \emph{isomorphism} or \emph{conjugacy} if $f$ is a homeomorphism and $\alpha$ is bijective. In this case we say that $(X,\Gamma)$ is \emph{isomorphic} or \emph{conjugate} to $(Y,\Lambda)$. 
	\end{enumerate}
		Morphisms may be composed by setting $(f,\alpha) \circ (g,\beta) = (f \circ g, \alpha \circ \beta)$.
\end{dfn}

	In Definition~\ref{dfn:morphism} we have not made use of the metric space structure of either $X$ or $Y$, and this has its drawbacks. In particular, if  $(f,\alpha) \colon (X,\Gamma) \to (Y,\Lambda)$ is a morphism, then it is not typically true that $(f(X),\alpha(\Gamma))$ is an iterated function system since $f(X)$ is not necessarily a complete metric space. 	This could be amended by insisting that $f$ is also a closed map or that $X$ is compact (for example if $X$ is the attractor itself), however we do not require either assumption in what follows.

Morphisms in the sense of Definition~\ref{dfn:morphism} occur fairly naturally.
\begin{example}
	Let $\Omega_N := \{1,\ldots,N\}^{\NN}$. For each $w \in \Omega_N$ we write $w = w_1 w_2 w_3 \cdots$, where each $w_i \in \{1,\ldots,N\}$. Equip $\Omega_N$ with the metric $d \colon \Omega_N \times \Omega_N \to [0,\infty)$ defined by
	\[
	d(w,v) = \begin{cases}
		2^{1-\min\{k \colon w_k \ne v_k\}} & \text{if } w \ne v, \\
		0 &\text{if } w = v, \\
	\end{cases}
	\]
	for each $w,v \in \Omega_N$, so that $\Omega_N$ is a Cantor space. For each $1 \le i \le N$ consider the contraction $\sigma_i \colon \Omega_N \to \Omega_N$ defined by $
	\sigma_i (w_1w_2\cdots) = iw_1w_2\cdots.
	$
	Then $(\Omega_N, \Sigma_N := \{\sigma_1,\ldots,\sigma_N\})$ is a hyperbolic iterated function system. Since $\Omega_N$ is invariant under the Hutchinson operator $\Sigma_N$, it is necessarily the attractor.
	
	Now suppose that $(X,\Gamma)$ is an iterated function system with attractor $\AA$, and suppose that $L \colon \{1,\ldots,N\} \to \Gamma $ is a bijective labelling of the maps in $\Gamma$. Denoting $L(i)$ by $\gamma_i$, there is a continuous surjection $\pi_L \colon \Omega_{N} \to \AA$ called the \emph{code map associated to the labelling $L$} which satisfies $\pi_L \circ \sigma_i = \gamma_i \circ \pi_L$ for all $1 \le i \le N$, \cite[Theorem 3.2]{Hat85}. The code map may be defined explicitly by
	\[
	\{\pi_L(w_1w_2\cdots)\} = \bigcap_{k=1}^\infty \gamma_{w_1} \circ \cdots \circ \gamma_{w_k} (\AA).
	\]
	Let $\beta_L \colon \Sigma_N \to \Gamma$ denote the bijection $\beta_L(\sigma_i) = \gamma_i$. Then $(\pi_L,\beta_L) \colon (\Omega_{N} , \Sigma_{N}) \to (\AA,\Gamma)$ is a semiconjugacy.	
\end{example}

\begin{example}
	Subsystems of iterated function systems give examples of embeddings. 
	Indeed, if $(X,\Lambda)$ is an iterated function system and $\Gamma \subseteq \Lambda$, then the pair $(\id_X, \Gamma \hookrightarrow \Lambda)$ is an embedding. 
\end{example}

\begin{example}\label{ex:edgeofgasket}	
	Consider the iterated function system $(\RR,\Gamma = \{\gamma_1,\gamma_2\})$ with $\gamma_1(x) = \frac{x}{2}$ and $\gamma_2(x) = \frac{1+x}{2}$. Then $[0,1]$ is the attractor of $(\RR,\Gamma)$.
	Now let $(\RR^2,\Lambda = \{\lambda_1, \lambda_2,\lambda_3\})$ with $\lambda_1 (x,y) = (\frac{x}{2},\frac{y}{2})$ , $\lambda_2 (x,y) = (\frac{1+x}{2},\frac{y}{2})$, and $\lambda_3(x,y) = (\frac{2x+1}{4} , \frac{2y+\sqrt{3}}{4})$. The attractor of $(\RR^2,\Lambda)$ is a Sierpinski gasket with outer vertices at $(0,0),(1,0)$ and $(\frac{1}{2}, \frac{\sqrt{3}}{2})$. Let $f \colon \RR \to \RR^2$ be the closed embedding given by $f(x) = (\frac{x}{2},\frac{\sqrt{3}x}{2})$ and let $\alpha \colon \Gamma \to \Lambda$ be given by $\alpha(\gamma_1) = \lambda_1$ and $\alpha(\gamma_2) = \lambda_3$. Then $(f,\alpha) \colon (\RR,\Gamma) \to (\RR^2,\Lambda)$ is an embedding. 
\end{example}

In the previous examples $\alpha$ was an injection, but this is not always the case.
\begin{example}
	Let $(X,\Gamma)$ be an iterated function system and consider the product system $(X^2,\Gamma^2 = \{\gamma \times \gamma' \mid \gamma,\gamma'\in \Gamma\})$, where $\gamma \times \gamma' (x,y) = (\gamma(x),\gamma'(y))$. The pair $((x,y) \mapsto x, \gamma \times \gamma' \mapsto \gamma)$ defines a semiconjugacy. 
\end{example}	
	
Morphisms intertwine the induced dynamics on compact sets.
\begin{lem}[{cf. \cite[Proposition 4.6.4 (vi)]{Kie02}}]\label{lem:morphisms_hyperspace}
	Let $(f,\alpha) \colon (X,\Gamma) \to (Y,\Lambda)$ be a morphism and consider the $d_H$-continuous map $\ol{f} \colon \Kk(X) \to \Kk(Y)$ induced by $f$. Then $\ol{f}$ is a homomorphism between the single-map dynamical systems $(\Kk(X),\Gamma)$ and $(\Kk(Y),\alpha(\Gamma))$ in the sense that $\ol{f} \circ \Gamma = \alpha(\Gamma) \circ \ol{f}$. Here, $\alpha(\Gamma)$ is the Hutchinson operator for the system $(Y,\alpha(\Gamma))$
\end{lem}
\begin{proof}
	For $K \in \Kk(X)$ we simply compute
	\begin{equation*}\label{eq:back_invar}
		\ol{f}(\Gamma(K)) = \bigcup_{\gamma \in \Gamma} f \circ \gamma(K) = \bigcup_{\gamma \in \Gamma} \alpha(\gamma) \circ f(K) = \alpha(\Gamma) (\ol{f}(K)). \qedhere
	\end{equation*}
\end{proof}
\begin{rmk}
	A morphism $(\pi,\alpha) \colon (X,\Gamma) \to (Y,\Lambda)$ is distinct from a homomorphism between the single-map systems $(\Kk(X),\Gamma)$ and $(\Kk(Y),\alpha(\Gamma))$. Indeed, if $X = Y$ and $\Gamma(K) = \Lambda(K)$ for all $K \in \Kk(X)$, then as a single-map system $(\Kk(X),\Gamma)$ is equal, not just conjugate, to $(\Kk(X),\Lambda)$.
	
	For example, fix $\gamma \in \Gamma$ and let $\gamma_0$ be a distinct copy of $\gamma$. Form the disjoint union $\Gamma_0 = \Gamma \sqcup \{\gamma_0\}$ so that $\gamma_0$ is a ``redundant'' map. Then $(\id_X,\Gamma \hookrightarrow \Gamma_0)$ is a morphism that is not a conjugacy as $\Gamma \hookrightarrow \Gamma_0$ does not surject, but the single-map systems $(\Kk(X),\Gamma)$ and $(\Kk(Y),\Gamma_0)$ are equal. 
\end{rmk}	
	
For hyperbolic systems, morphisms always map attractors to compact subsets of attractors so that the image is a subsystem of the codomain.
	 
	 \begin{lem}\label{lem:imageofAinB}
	 	Let $(X,\Gamma)$ and $(Y,\Lambda)$ be  iterated function systems with attractors $\AA$ and $\BB$, respectively. If $(f,\alpha) \colon (X,\Gamma) \to (Y,\Lambda)$ is a morphism, then $f(\AA) \subseteq \BB$. In particular, $(f(\AA),\alpha(\Gamma))$ embeds in $(\BB,\Lambda)$.
	 \end{lem}
	 
	 \begin{proof}
	 Lemma~\ref{lem:morphisms_hyperspace} implies that $f(\AA)$ is backward invariant as $f(\AA) = f(\Gamma(\AA)) = \alpha(\Gamma)(f(\AA)) \subseteq \Lambda(f(\AA))$.  Since $f(\AA)$ is compact it follows from Lemma~\ref{lem:backinvariant} that $f(\AA) \subseteq \BB$. 
	 \end{proof}

\section{A closed graph theorem for morphisms}

In this section we prove the main result of this article, a Closed Graph Theorem for morphisms of hyperbolic iterated function systems. We show that, when restricted to attractors, the graph of a morphism is itself the attractor of an iterated function system. To this end, we introduce a fibred system.
\begin{dfn} \label{dfn:fibred_system}
	Let $(X,\Gamma)$ and $(Y,\Lambda)$ be hyperbolic iterated function systems and suppose that $\alpha \colon \Gamma \to \Lambda$. For each $\gamma \in \Gamma$ define $\gamma^{\alpha} \colon X \times Y \to X\times Y$ by	
	\[
	\gamma^{\alpha}(x,y) = (\gamma(x),\alpha(\gamma)(y))
	\]
	and let $\Gamma \times_{\alpha} \Lambda := \{\gamma^{\alpha} \mid \gamma \in \Gamma\}.$
	Equipping $X \times Y$ with the metric $d_{\infty} ((x_1,y_1),(x_2,y_2)) = \max\{d_X(x_1,x_2), d_Y(y_1,y_2)\}$ the pair $(X \times Y , \Gamma \times_{\alpha} \Lambda)$ is a hyperbolic iterated function system which we call the \emph{system fibred over $\alpha$\footnote[2]{Any metric for which  $(X \times Y , \Gamma \times_{\alpha} \Lambda)$ is hyperbolic may be used.}}.
\end{dfn}

When $X = \Omega_N$ and $|\Lambda| = N$ the system fibred over $\alpha$ is referred to as the \emph{lifted} system by Barnsley in which it is used to describe fractal tops \cite[Definition 4.9.1]{Bar06}. We now come to the main result. 

\begin{thm}	\label{thm:graphattractor}
	Let $(X,\Gamma)$ and $(Y,\Lambda)$ be hyperbolic iterated function systems with attractors $\AA$ and $\BB$, respectively.
	If $(f,\alpha) \colon (X,\Gamma) \to (Y,\Lambda)$ is a morphism, then the attractor of the system 
	$(X \times Y , \Gamma \times_{\alpha} \Lambda)$ is the graph 
	\[
	\gr(\restr{f}{\AA}) := \{ (x,f(x)) \in X \times Y \mid x \in \AA\}
	\]
	 of $f$ restricted to $\AA$.  
	 Moreover, suppose that $\beta \colon \Gamma \to \Lambda$ and $\DD$ is the attractor of $(X \times Y , \Gamma \times_{\beta} \Lambda)$. Then there is a continuous function $g \colon \AA \to \BB$ making $(g,\beta) \colon (\AA,\Gamma) \to (\BB,\Lambda)$ a morphism if and only if $\DD$ is the graph of a function from $\AA$ to $\BB$ (continuity is automatic).
\end{thm}

\begin{proof}
	
		Lemma~\ref{lem:imageofAinB} implies that $f(\AA) \subseteq \BB$, so $\gr(\restr{f}{\AA})$ is a subset of the compact set $\AA \times \BB$. 	
	 Since $f$ is continuous, $\gr(\restr{f}{\AA})$ 
	 is closed in $\AA \times \BB$ and therefore compact.
	 The iterated function system $(X \times Y , \Gamma \times_{\alpha} \Lambda)$  has a unique attractor by Hutchinson's Theorem.
	Hence, it suffices to show that
	\[
	\gr(\restr{f}{\AA}) = \bigcup_{\gamma \in \Gamma} \gamma^{\alpha}(\gr(\restr{f}{\AA})).
	\]			
	Since $f \circ \gamma = \alpha(\gamma) \circ f$ it follows that if $(x,f(x)) \in \gr(\restr{f}{\AA})$, then $\gamma^{\alpha}(x,f(x)) \in \gr(\restr{f}{\AA})$ for all $\gamma \in \Gamma$.	
	For the reverse inclusion fix $(x,f(x)) \in \gr(\restr{f}{\AA})$. Since $\AA = \bigcup_{\gamma \in \Gamma} \gamma(\AA)$, for each $x \in \AA$ there exists $\gamma \in \Gamma$ and $x_0 \in \AA$ such that $x = \gamma(x_0)$. Then
	\[
	\gamma^{\alpha}(x_0,f(x_0)) 
	= (\gamma(x_0),\alpha(\gamma) \circ f(x_0)) 
	= (x,f\circ \gamma (x_0)) 
	= (x,f(x)),
	\]
	so $(x,f(x)) \in 	\gamma^{\alpha}(\gr(\restr{f}{\AA}))$. As such, $\gr(\restr{f}{\AA})$ is the attractor of $(X \times Y , \Gamma \times_{\alpha} \Lambda)$.
	
	The ``only if'' direction of the second statement follows from the first statement. For the ``if'' direction suppose that $\DD$ is the graph of a function $g \colon \AA \to \BB$.
	Since $\DD$ is closed in $\AA \times \BB$ it follows from the Closed Graph Theorem for compact Hausdorff spaces that $g$ is continuous. If $(x,g(x)) \in \DD$, then invariance under the Hutchinson operator implies that for each $\gamma \in \Gamma$, we have $(\gamma(x),\beta(\gamma) \circ g(x))  \in \DD$. Since $\DD = \gr(g)$ it follows that $\beta(\gamma) \circ g (x) = g \circ \gamma(x)$. Consequently, $(g,\beta) \colon (\AA,\Gamma) \to (\BB,\Lambda)$ is a morphism.
\end{proof}

	Restricting to attractors yields the following Closed Graph Theorem. 
	\begin{cor}[Closed Graph Theorem]\label{cor:closedgraph}
		Let $(\AA,\Gamma)$ and $(\BB,\Lambda)$ be hyperbolic iterated function systems with attractors $\AA$ and $\BB$, respectively. Suppose $f \colon \AA \to \BB$ and $\alpha \colon \Gamma \to \Lambda$. Then $(f,\alpha)$ is a morphism if and only if $\gr(f)$ is the attractor of $(X \times Y, \Gamma \times_\alpha \Lambda)$. 
	\end{cor}
\begin{rmk}
	Set theoretically, a function \emph{is} its graph. Consequently, Corollary~\ref{cor:closedgraph} may be interpreted as saying that $f$ \emph{is} the attractor of an iterated function system. 
\end{rmk}


Theorem~\ref{thm:graphattractor} also implies that morphisms between hyperbolic iterated function systems are rare, and completely determined by $\alpha$ on the attractor. 
\begin{cor}[Morphism rigidity]\label{cor:morphismsrigid}
	If $(f,\alpha) \colon (X,\Gamma) \to (Y,\Lambda)$ is a morphism between hyperbolic iterated function systems, then $\restr{f}{\AA}$ is determined entirely by $\alpha$. 	
	In particular, if $(g,\alpha)\colon (X,\Gamma) \to (Y,\Lambda)$ is another morphism, then $\restr{f}{\AA} = \restr{g}{\AA}$.
\end{cor}
\begin{proof}
	Theorem~\ref{thm:graphattractor} implies that $\gr(f)$ and $\gr(g)$ are both attractors of $(X \times Y , \Gamma \times_{\alpha} \Lambda)$, so uniqueness of the attractor gives the result. 
\end{proof}

	In light of Corollary~\ref{cor:morphismsrigid} we may use Theorem~\ref{thm:graphattractor} to approach the problem of determining whether there exists a morphism between two hyperbolic iterated function systems. If $(f,\alpha) \colon (X,\Gamma) \to (Y,\Lambda)$ is a morphism, then so is $(\restr{f}{\AA},\alpha) \colon (\AA,\Gamma) \to (\BB,\Lambda)$. Since $\restr{f}{\AA}$ is determined completely by $\alpha$, it suffices to check whether the attractor of $(X \times Y , \Gamma \times_{\alpha} \Lambda)$ is the graph of a function for each of the $|\Lambda|^{|\Gamma|}$ possible choices of $\alpha$. Moreover, if $|\Gamma| = |\Lambda|$ we can determine whether a conjugacy exists by checking each of the $|\Gamma|!$ possible choices of $\alpha$. 
	
	For concrete iterated function systems, the attractor of  $(X \times Y , \Gamma \times_{\alpha} \Lambda)$ may be approximated numerically using the Chaos Game algorithm. This can inform existence or non-existence results about morphisms or conjugacy as seen in Example~\ref{ex:intervalmaps} below.

\begin{example} \label{ex:intervalmaps}
	Consider the unit interval $[0,1]$  with the Euclidean metric, and define iterated function systems $([0,1], \Gamma=\{\gamma_1,\gamma_2\})$ and $([0,1], \Lambda = \{\lambda_1,\lambda_2\})$, where
	\[
	\gamma_1(x) = \frac{2x}{3}, \qquad \gamma_2(x) = \frac{2x}{3} + \frac{1}{3}, \qquad \lambda_1(x) = \frac{3x}{4}, \text{ and} \qquad \lambda_2(x) = \frac{3x}{4} + \frac{1}{4}.
	\]
	At a first glance, the systems $([0,1], \Gamma)$ and $([0,1],\Lambda)$ exhibit a similar behaviour. Both have attractor $[0,1]$, and the sets of overlap are closed intervals given by $\gamma_1([0,1]) \cap \gamma_2([0,1]) = [\frac{1}{3}, \frac{2}{3}]$ and $\lambda_1([0,1]) \cap \lambda_2([0,1]) = [\frac{1}{4}, \frac{3}{4}] $, respectively. It is natural to ask whether these two systems are conjugate. 
	
	In Figure~\ref{fig:intervalmaps} a Chaos Game approximation for the attractor $\DD$ of  $([0,1]^2,\Gamma \times_\alpha \Lambda)$ is pictured for the bijection $\alpha \colon \Gamma \to \Lambda$ given by $\alpha(\gamma_1) = \lambda_1$ and $\alpha(\gamma_2) = \lambda_2$. The attractor of the other bijection is given by a horizontal reflection of Figure~\ref{fig:intervalmaps}. The approximation makes it easy to see that the systems $([0,1],\Gamma)$ and $([0,1],\Lambda)$ are not conjugate as $\DD$ is not the graph of a bijection. 
	\begin{figure}[ht]
		\centering
		\begin{tikzpicture}			
			[vertex/.style={circle, fill=red, inner sep=1pt}]
			
			\def\scale{4}
			
			\node[anchor=south west,inner sep=0] at (0,0) {\includegraphics[width=\scale cm]{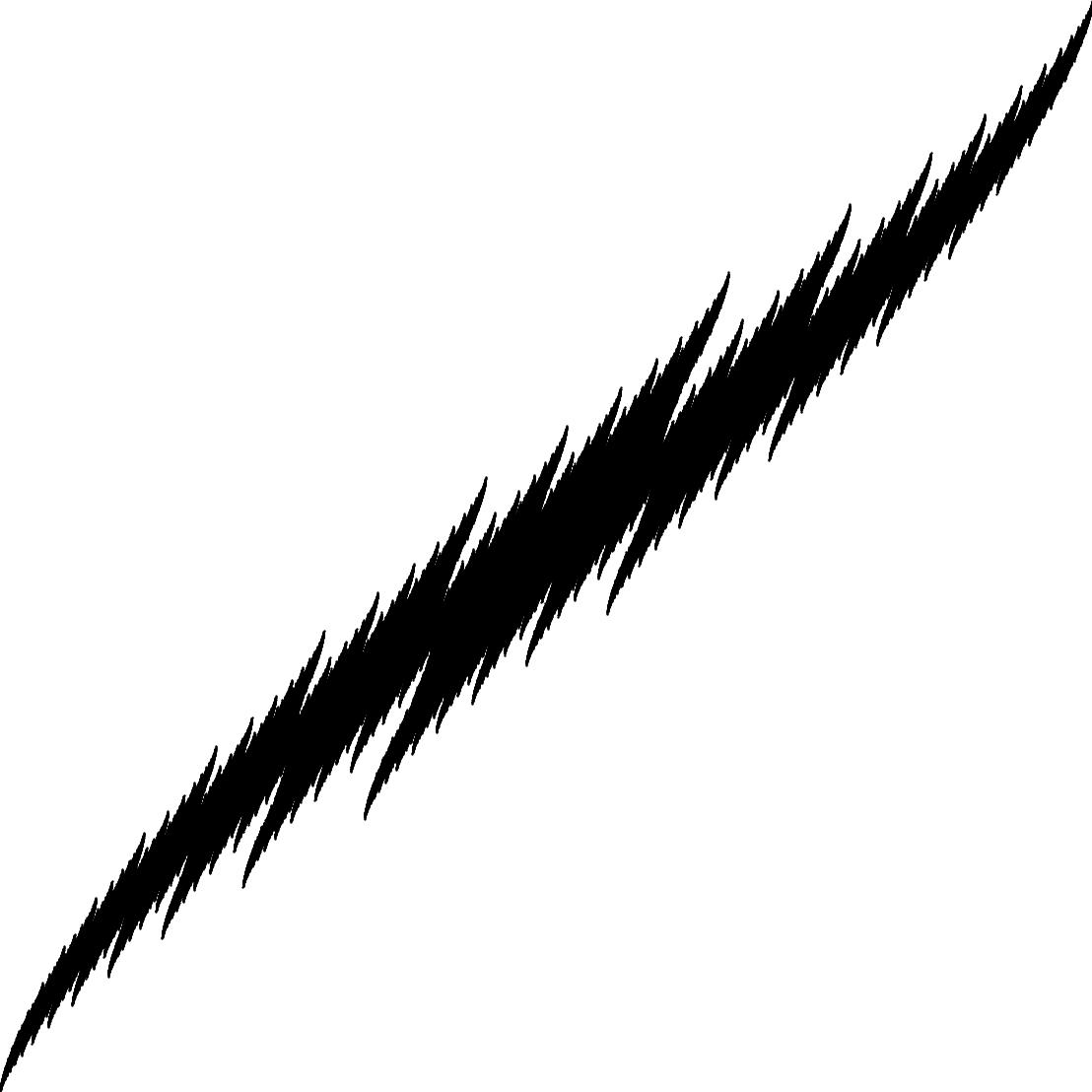}};
			 
			 \node[vertex, label={[label distance=0.1cm]180:\scriptsize$(0,0)$}] (a) at (0,0) {};
			 \node[vertex, label={[label distance=0.1cm]0:\scriptsize$(1,1)$}] (b) at (\scale,\scale) {};
			 \node[vertex, label={[label distance=0.1cm]180:\scriptsize$(\frac{4}{9},\frac{9}{16})$}] (c) at (4*\scale/9,9*\scale/16) {};
			 \node[vertex, label={[label distance=0.1cm]0:\scriptsize$(\frac{5}{9},\frac{7}{16})$}] (d) at (5*\scale/9,7*\scale/16) {};
			
		\end{tikzpicture}
		\caption{A Chaos Game approximation of the attractor $\DD$ of the system $([0,1]^2, \Gamma \times_\alpha \Lambda)$ from Example~\ref{ex:intervalmaps}.}
		\label{fig:intervalmaps}
	\end{figure}
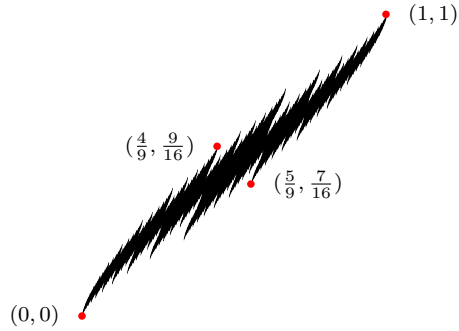
	Formalising the observation, $(0,0)$ and $(1,1)$ belong to $\DD$ as the respective fixed points of $\gamma_1^\alpha$ and $\gamma_2^\alpha$. Then $(\frac{5}{9},\frac{7}{16}) = \gamma_2^\alpha \circ \gamma_2^\alpha(0,0)$ and $(\frac{4}{9},\frac{9}{16}) = \gamma_1^\alpha \circ \gamma_1^\alpha (1,1)$ also belong to $\DD$. If $\DD$ were the graph of a (necessarily continuous) function $f$, then $f(0) =0$, $f(1) = 1$, $f(\frac{4}{9}) = \frac{9}{16}$, and $f(\frac{5}{9}) = \frac{7}{16}$. The Intermediate Value Theorem implies that $f$ could not be injective.
\end{example}

 Corollary~\ref{cor:closedgraph} also gives an alternative definition of the code map.
\begin{example} \label{ex:codemapfractal}
	Consider an iterated function system $(\AA,\Gamma)$ with attractor $\AA$. Let $L \colon \{1,\ldots,N\} \to \Gamma$ be a bijective labelling of the maps in $\Gamma$ and let $\beta_L \colon \Sigma_N \mapsto \Gamma$ denote the induced map as in Example~\ref{ex:codemapfractal}. Then the code map $\pi_L \colon \Omega_N \to \AA$ is uniquely determined as the function whose graph is the attractor of the fibred system $(\Omega_N \times \AA, \Sigma_N \times_{\beta_L} \Gamma)$.

\end{example}

Restricting to attractors, the fibred system associated to a morphism is always conjugate to the domain. 
\begin{cor}\label{cor:graph_conjugate}
	Suppose that $(f,\alpha) \colon (\AA,\Gamma) \to (\BB,\Lambda)$ is a morphism of hyperbolic iterated function systems  with respective attractors $\AA$ and $\BB$. Then $(\gr(f), \Gamma \times_\alpha\Lambda)$ is conjugate to $(\AA,\Gamma)$. 
\end{cor}
\begin{proof}
	Theorem~\ref{thm:graphattractor} implies that $\gr(f)$ is the attractor of $(\gr(f), \Gamma \times_\alpha\Lambda)$.	Let $p \colon \gr(f) \to \AA$ denote the projection onto the first factor. Since $\gr(f)$ is the graph of a function, $p$ is a continuous bijection from a compact space to a Hausdorff space, and therefore a homeomorphism. It follows that $(p,\gamma^\alpha \mapsto \gamma)$ is the desired conjugacy. 
\end{proof}

As a consequence of Corollary~\ref{cor:graph_conjugate}, topological properties of the system $(\AA,\Gamma)$ are shared by $(\gr(f), \Gamma \times_\alpha\Lambda)$. For example, the covering dimension of $\AA$ is equal to the covering dimension of $\gr(f)$. The relationship at the level of geometric properties is not so clear, because geometric properties of $\gr(f)$ depend on the choice of metric on $\gr(f)$.


We finish by showing that morphisms of of iterated function systems lift uniquely to morphisms between code spaces.
\begin{prop}
	Let $(\AA,\Gamma=\{\gamma_1,\ldots,\gamma_N\})$ and $(\BB,\Lambda = \{\lambda_1,\ldots,\lambda_M\})$ be hyperbolic iterated function systems with chosen labellings of the contractions, and attractors $\AA$ and $\BB$, respectively. Let $(\pi_\Gamma,\beta_\Gamma) \colon (\Omega_N,\Sigma_N) \to (\AA,\Gamma)$ and $(\pi_\Lambda, \beta_\Lambda) \colon (\Omega_N,\Sigma_N) \to (\BB,\Lambda)$ denote the morphisms induced by the corresponding code maps. For any morphism $(f,\alpha) \colon (\AA,\Gamma) \to (\BB,\Lambda)$ there exists a unique morphism $(\widetilde{f},\widetilde{\alpha}) \colon (\Omega_N,\Sigma_N) \to (\Omega_M,\Sigma_M)$ making the diagram
	\begin{equation}\label{eq:code_cd}
		\begin{tikzcd}
			{(\Omega_N,\Sigma_N)} \arrow[r, "{(\widetilde{f},\widetilde{\alpha})}"] \arrow[d, "{(\pi_\Gamma,\beta_\Gamma)}"'] & {(\Omega_M,\Sigma_M)} \arrow[d, "{(\pi_\Lambda,\beta_\Lambda)}"] \\
			{(\AA,\Gamma)} \arrow[r, "{(f,\alpha)}"]                                                                          & {(\BB,\Lambda)}                                                 
		\end{tikzcd}
	\end{equation}
	commute.
\end{prop}

\begin{proof}
	Since $\beta_\Lambda$ is a bijection, we can define $\widetilde{\alpha} := \beta_\Lambda^{-1} \circ \alpha \circ \beta_\Gamma$. Then $\widetilde{\alpha}$ induces a function $h \colon \{1,\ldots, N\} \to \{1, \ldots, M\}$ defined by $\widetilde{\alpha} ( \sigma_i) = \sigma_{h(i)}$ and $h$ induces a continuous map $\widetilde{f} \colon \Omega_N \to \Omega_M$ defined by $\widetilde{f}(w_1w_2w_3 \cdots) = h(w_1)h(w_2)h(w_3) \cdots$ for all $w_k \in \{1, \ldots, N\}$. In particular, $\widetilde{f} \circ \sigma_i = \widetilde{\alpha}(\sigma_i) \circ \widetilde{f}$ for all $i \in \{1 ,\ldots, N\}$. Uniqueness of $\widetilde{f}$ follows from Corollary~\ref{cor:closedgraph}. 
\end{proof}
\begin{rmk}
	It is not the case that every morphism $(\widetilde{f},\widetilde{\alpha}) \colon (\Omega_N,\Sigma_N) \to (\Omega_M,\Sigma_M)$ descends to a morphism $(f,\alpha) \colon (\AA,\Gamma) \to (\BB,\Lambda)$ making $\eqref{eq:code_cd}$ commute. Indeed, if $M = N$, $(\BB,\Lambda) = (\Omega_N,\Sigma_{N})$, and  $(\widetilde{f},\widetilde{\alpha}) = (\pi_\Lambda,\beta_\Lambda)=(\id_{\Omega_N},\id_{\widetilde{\Sigma_N}})$, then commutativity of \eqref{eq:code_cd} would imply that such an $f$ is an inverse for $\pi_\Gamma$---regardless of $\AA$---which is absurd.
\end{rmk}

As a final remark, we note that the notion of morphism introduced in Definition~\ref{dfn:morphism} generalises readily to the topological iterated function systems of Kameyama \cite{Kam93} or Kieninger \cite{Kie02}.
Although we do not pursue it here, the author thinks that it would be interesting to see how the collection of invariant sets in the fibred system affects the existence of morphisms for more general topological iterated function systems. 
\section{References}

\printbibliography

\bigskip{\footnotesize%
	\textsc{School of Mathematics and Applied Statistics,
		University of Wollongong, Wollongong,
		NSW  2522, Australia} \par
	\textit{E-mail:} \texttt{amundey@uow.edu.au} 
}

\end{document}